\documentclass[10pt]{article}
\usepackage{geometry}
\usepackage{amsmath}
\usepackage{amssymb}

\usepackage{enumerate}

\usepackage{listings}
\usepackage{color}
\definecolor{codegreen}{rgb}{0,0.6,0}
\definecolor{codepurple}{rgb}{0.58,0,0.82}
\definecolor{codered}{rgb}{0.7,0,0}
\lstdefinestyle{mystyle}{ 
	commentstyle=\color{codegreen},
	numberstyle=\tiny,
	stringstyle=\color{codepurple},
	basicstyle=\ttfamily\footnotesize,
	breakatwhitespace=false,         
	breaklines=true,                 
	captionpos=b,                   
	keepspaces=true,
	keywordstyle=\color{blue},                                
	showspaces=false,                
	showstringspaces=false,
	showtabs=false,                  
	tabsize=2,
	numbers=left,
	stepnumber=1,
	emph=[1]{feval},emphstyle=[1]\color{black}
}
\usepackage{fancyhdr}
\usepackage[hidelinks]{hyperref} 

\usepackage{amsthm} 
\newtheorem{theorem}{Theorem}
\numberwithin{theorem}{section}
\newtheorem*{theorem*}{Theorem}

\newtheorem{lemma}[theorem]{Lemma}

\newtheorem{remark}[theorem]{Remark}

\newtheorem*{example*}{Example}

\usepackage{mathrsfs}

\usepackage{xcolor} 

\usepackage{caption}
\usepackage{subcaption}
\usepackage{graphicx}

\usepackage{algorithm}
\usepackage{algpseudocode} 

\usepackage{bbm}

\usepackage{bm}

\usepackage{mathtools}

\usepackage[toc,page]{appendix}

\usepackage{caption}
\usepackage{subcaption}

\title{Stable Hermite transforms via the Golub--Welsch algorithm}
\author{Marcus Webb\footnote{Corresponding author, \href{mailto:marcus.webb@manchester.ac.uk}{marcus.webb@manchester.ac.uk}} \\ Department of Mathematics\\ The University of Manchester
\\ \,\\
Georg Maierhofer \\ Department of Applied Mathematics and Theoretical Physics\\ University of Cambridge}
\date{\today}
\begin{document}

\maketitle

\begin{abstract}
We introduce an efficient and stable algorithm for transforms associated with expansions in Hermite functions interpolated at Hermite polynomial roots. The Hermite transform matrix can be factorised into a diagonal component and an orthogonal matrix, leading to a form which allows both the forward and inverse Hermite transforms to be computed stably. Our novel algorithm computes this factorisation based on the eigendecomposition of the Jacobi matrix associated with Hermite functions. Through numerical experiments, we demonstrate the stability and efficiency gains of this novel method over prior work. Numerical experiments show that the new approach matches or improves on the accuracy of existing stabilized methods, is substantially faster in practice, and enables reliable use of large Hermite expansions in downstream PDE computations. We also provide an open-source implementation, together with reference implementations of previous methods, to facilitate adoption by the community.
\end{abstract}
\paragraph{Key words:} Spectral Methods; Hermite Functions; Computational Quantum Mechanics
\paragraph{Mathematics Subject Classification (MSC2020):} 65M70; 65D20; 65D15

\section{Introduction}
In the numerical solution of certain PDEs, such as the semi-classical Schr\"odinger equation, Hermite spectral methods, in which the solution is expanded as a series in Hermite functions are a natural method used by many researchers. As we will show in this paper, the standard algorithms used in several papers, such as \cite{ThCaNe2009,gauckler2011convergence,bao2005fourth, Thalhammer2012}, encounter underflow or overflow in double precision arithmetic when the polynomial degree of the expansion is taken to be greater than $766$. Indeed, \cite{bao2009generalized} says ``...care must be taken when computing the weights for large $N$." In some settings extremely large values of $N$ are necessary to accurately resolve the solution. An example of this is the semi-classical Schr\"odinger equation, in which the ratio between the mass of an electron and the mass of an atomic nucleus leads to oscillatory solutions requiring a large number of Hermite modes to represent these oscillations (this follows from WKB analysis, cf.~\cite{jin2011mathematical}). In this paper we discuss numerically stable algorithms for Hermite spectral methods that avoid this numerical issue, one of which is new, the other is not widely known, but appeared in \cite[Appendix]{bunck2009fast}.

\subsection{Hermite transforms}

Hermite functions appear as fundamental special functions in a range of fields, including in quantum physics \cite{griffiths2018introduction}, statistical physics \cite{grohs2017tensor,WANG2019108815}, probability theory and statistics \cite{walter1977properties,izenman1991review}, and corresponding numerical methods \cite{Thalhammer2012,ThCaNe2009,gauckler2011convergence,amparoseguratemme2007,olver2020fast, iserles2023approximation}. They are defined as follows
\begin{align}\label{eqn:def_hermite_fns}
    \psi_n(x) = (2^n n! \sqrt{\pi})^{-\frac12} \mathrm{e}^{-\frac{x^2}{2}} \mathrm{H}_n(x),
\end{align}
where $\mathrm{H}_n(x)$ is the degree $n$ Hermite polynomial, defined by the recurrence $\mathrm{H}_0(x) = 1$, $\mathrm{H}_1(x) = 2x$,
\begin{align*}
    \mathrm{H}_{n+1}(x) = 2x\mathrm{H}_n(x) - 2n\mathrm{H}_{n-1}(x), \qquad n = 1,2,\ldots.
\end{align*}
The Hermite polynomials are orthogonal polynomials on $L^2(\mathbb{R},\mathrm{e}^{-x^2})$ and the Hermite functions $\{\psi_n\}_{n\in\mathbb{N}}$ form an orthonormal basis of $L^2(\mathbb{R})$ and satisfy the recurrence
\begin{equation}\label{eqn:hermiterec}
    \psi_{n+1}(x) = \sqrt{\frac{2}{n+1}}x\psi_n(x) - \sqrt{\frac{n}{n+1}}\psi_{n-1}(x).
\end{equation}
This paper concerns the expansion of functions in terms of this Hermite function basis:
\begin{equation}\label{eqn:finite_expansion}
    f(x) \approx f_N(x) = \sum_{n=0}^{N{-}1} c_n \psi_n(x).
\end{equation}

Let $x_0,\dots, x_{N{-}1}\in\mathbb{R}$ be the Gauss--Hermite quadrature nodes of degree $N$ (i.e.~the roots of $\psi_N$). We are solely concerned with the function $f_N \in \mathrm{span}\{\psi_0,\ldots,\psi_{N{-}1} \} = :\mathcal{H}_N$ that interpolates $f$ at the nodes.

If we are given a vector of coefficients $\mathbf{c} = (c_0,\ldots,c_{N{-}1})^T$, we can compute the vector of values $\mathbf{v} = (f_N(x_0),\ldots,f_N(x_{N{-}1}))^T$ by the matrix-vector multiplication $\mathbf{v} = T\mathbf{c}$, where ${T}\in\mathbb{R}^{N\times N}$ with
\begin{equation}\label{eqn:T}
    {T}_{ij}=\psi_{j}(x_i), \quad 0\leq i,j\leq N{-}1.
\end{equation}
This is called a forward discrete transform (or synthesis operation \cite{olver2020fast}). The backward discrete transform (or analysis operation) is $\mathbf{c} = T^{-1}\mathbf{v}$. This analysis operation can be used to take the values of $f$ at the nodes and compute the coefficients of the unique interpolant in $\mathcal{H}_N$.

The computational task considered in this paper is that of building the matrices $T$ and $T^{-1}$ for a given $N$. In tandem, these two matrices can be used to transform forwards and backwards between coefficients and values, performing the necessary operations in a Hermite spectral collocation method (see section \ref{sec:downstream}). We are interested only in algorithms requiring $\mathcal{O}(N^2)$ operations to perform this task and the novel aspects in this paper are in ways to ensure that $T$ and $T^{-1}$ are computed in a numerically stable manner.

In practice, the matrix $T$ is often computed directly using the recurrence in \eqref{eqn:hermiterec} (cf.~\cite{Thalhammer2012,ThCaNe2009,gauckler2011convergence,bao2005fourth, bao2009generalized}). In double precision this method encounters underflow for $N \geq 766$, which limits the applicability of the method. {In principle, \cite{leibon2008fast} introduced a fast algorithm for the assembly of $T$ in $\mathcal{O}(N\log(N))$ operations. However, unlike fast transforms for Jacobi polynomial expansions (cf. \cite{townsend2018fast, slevinsky2018use, olver2020fast}), this fast Hermite transform suffers from instabilities for $N \geq 64$ (see \cite[Table 2]{leibon2008fast}).}

This stability problem of computing $T$ was discussed by Bunck in \cite{bunck2009fast}, in response to which he introduced a stable recurrence relation for computing Hermite functions, which has, however, only received limited attention from the community. In recent years, this algorithm has been successfully used to design a stable Hermite transform in the Julia packages \verb|FastTransforms.jl| and \verb|FastGaussQuadrature.jl| \cite{FastTransforms,FastGaussQuadrature}. This stable approach has been also been successfully applied in the context of the computation of dispersive partial differential equations \cite{banicamaierhoferschratz26,carlesmaierhofer25} posed on unbounded domains.

\subsection{Contributions and paper structure}

In this paper we propose an alternative method for calculating $T$ and $T^{-1}$, based on a by-product of the Golub--Welsch algorithm for calculating Gaussian quadrature rules \cite{golub1969calculation}. We produced a simple MATLAB implementation of the new method (see Appendix \ref{app:matlab_code}), and found that in practice it is as accurate as Bunck's stabilized recurrence, and as fast as the direct method (see section \ref{sec:numerical_examples}).

The remainder of this manuscript is structured as follows. In Section~\ref{sec:hermite_transform} we define the transformation between function values and expansion coefficients more precisely and describe the stability issue that appears when using a direct transformation. This is followed, in Section~\ref{sec:Bunck_method} with a description of Bunck's original method from a modern perspective and, in Section~\ref{sec:Golub-Welsch_method} with a description of our new algorithm whose performance we test in several numerical examples in Section~\ref{sec:numerical_examples}. Finally, we provide a MATLAB implementation of the algorithm, which can be found in Appendix~\ref{app:matlab_code} and is also publicly available here: \url{https://github.com/marcusdavidwebb/StableHermiteTransforms}.

\section{Algorithms for the Hermite transform}\label{sec:hermite_transform}

As discussed in the introduction, we can build the Hermite transform matrix $T$ in $\mathcal{O}(N^2)$ operations by taking advantage of the three-term recurrence. However, for $|x| > 39$, $\mathrm{e}^{-x^2 / 2}$ underflows to zero, so when $N \geq 766$ the contribution of $\psi_0(x_0)$ is lost, leading to catastrophic instability that can be seen in practice in Section~\ref{sec:numerical_examples}. 

To avoid this underflow problem, one could consider scaling the rows and columns and later ``de-scaling'' them, which will be discussed in section \ref{sec:Bunck_method}.

To compute the inverse transform matrix $T^{-1}$, we can use the exactness property of Gauss--Hermite quadrature. Namely, if $w_0,\ldots,w_{N{-}1}$ are the quadrature weights, then
\begin{align}\label{eqn:GHexact}
    \sum_{k=0}^{N-1} w_k \mathrm{e}^{x_k^2} \psi_i(x_k) \psi_j(x_k) = \int_{-\infty}^\infty \psi_i(x) \psi_j(x) \,d x = \delta_{i,j},
\end{align}
because $\mathrm{e}^{x^2} \psi_i(x) \psi_j(x)$ is a polynomial of degree at most $2N-2$. This implies that
\begin{equation}\label{eqn:Tinv}
    T^{-1} = T^T W,
\end{equation}
where $W$ is a diagonal matrix with elements $W_{kk} = w_k \mathrm{e}^{x_k^2}$. However, a na\"ive calculation of the matrix $W$ encounters underflow and overflow in $w_k$ and $\mathrm{e}^{x_k^2}$ respectively whenever $N \geq 372$. 

The calculation of $W$ can be stabilized as follows. The weights have the formula \cite[(15.3.6)]{szego1939orthogonal}
\begin{equation}
    w_k = \frac{\pi^{1/2}2^{N+1}N!}{\mathrm{H}_{N}'(x_k)^2}.
\end{equation}
The relation $\mathrm{H}'_N = 2N \mathrm{H}_{N{-}1}$ \cite[(18.9.25)]{NIST:DLMF} and equation \eqref{eqn:def_hermite_fns} lead to
\begin{equation}\label{eqn:ghweights}
    w_k = \frac{\mathrm{e}^{-x_k^2}}{N \psi_{N{-}1}(x_k)^2}.
\end{equation}
Therefore, $W_{kk} = N^{-1}\psi_{N{-}1}(x_k)^{-2}$, which can be extracted directly from the final column of $T$.

Nonetheless, for large $N$, an alternative algorithm is required for both $T$ and $T^{-1}$ that avoids underflow or overflow.

\subsection{Bunck's stabilized recurrence}\label{sec:Bunck_method}

\begin{algorithm}[h!]\caption{Bunck's stabilized recurrence for building $T$ \\ \textbf{Input:} Gauss--Hermite nodes $\mathbf{x} \in \mathbb{R}^{N}$ \\ \textbf{Output:}  $T\in\mathbb{R}^{N\times N}$ in equation \eqref{eqn:T}.}\label{alg:HermiteTransformPlan}
\begin{algorithmic}[1]
\State $\mathbf{h}_0 = \pi^{-1/4} \mathbf{1}$
\State $\mathbf{h}_1=2^{1/2}\pi^{-1/4} \mathbf{x}$
\For{$j= 1,2,3,\ldots,N{-}2$}
\State $\mathbf{h}_{j+1} = \sqrt{\tfrac{2}{j+1}} \mathbf{x} \odot \mathbf{h}_{j} - \sqrt{\tfrac{j}{j+1}} \mathbf{h}_{j-1}$
\State Define $\mathbf{s}_j \in \mathbb{R}^{N}$ by \begin{align*}
		(\mathbf{s}_{j})_k=\begin{cases}
			0,&\text{ if}\ \left|\left(\mathbf{h}_{j}\right)_k\right|< 10,\\
			\log\left|\left(\mathbf{h}_{j}\right)_k\right| ,& \text{ otherwise.}
		\end{cases}
	\end{align*}
    \State $\mathbf{h}_{j+1} = \mathbf{h}_{j+1} \odot \exp(-\mathbf{s}_j)$
    \State $\mathbf{h}_j = \mathbf{h}_j \odot \exp(-\mathbf{s}_j)$
\EndFor
\For{$j=1,2,\ldots,N{-}1$}
\State $\mathbf{h}_j = \mathbf{h}_j \odot \exp\left(-\frac{1}{2} \mathbf{x} \odot \mathbf{x}+\sum_{i=1}^j \mathbf{s}_i \right)$
\EndFor
\State $T = [\mathbf{h}_0, \mathbf{h}_1,\ldots,\mathbf{h}_{N{-}1}] $
\end{algorithmic}
\end{algorithm}

 The following is a description of a method proposed by Bunck in \cite[Appendix]{bunck2009fast}, of scaling the elements of the matrix $T$ as they are computed, to avoid underflow or overflow, and then undo the scaling at the end. 
 
The first alternative proposed by Bunck, to avoid the underflow in the first column of $T$ from propagating through the remaining columns, is to instead generate the vectors $\mathbf{h}_j\in\mathbb{R}^N$, with $$\left(\mathbf{h}_j\right)_k = \psi_j(x_k)\mathrm{e}^{x_k^2/2}, \text{ for } k = 0,\ldots, N{-}1$$ starting from $\mathbf{h}_0 = \pi^{-1/4} \mathbf{1}$ and $\mathbf{h}_1 = 2^{1/2}\pi^{-1/4} \mathbf{x}$, proceeding with the same recurrence and then multiplying by $\exp(-x_k^2/2)$ at the end. However, this leads to overflow in the $\mathrm{h}_j$ vectors.

The approach that Bunck proposed was to compute vectors $\mathbf{h}_j$ so that
\begin{equation}\label{eqn:Bunckscaling}
    \left(\mathbf{h}_j\right)_k = \psi_j(x_k)\exp\left(\frac{x_k^2}{2} - \sum_{i=1}^j \left(\mathbf{s}_i\right)_k \right),
\end{equation}
where $(s_j)_k$ is chosen adaptively so that $(h_j)_k = \mathcal{O}(1)$. Then at the end of the calculation, $\psi_j(x_k)$ can be extracted from equation \eqref{eqn:Bunckscaling}. Algorithm \ref{alg:HermiteTransformPlan} demonstrates this process explicitly.

\subsection{Orthogonal decomposition of the transform matrix}

Equation \eqref{eqn:Tinv} indicates that there is a hidden orthogonal matrix that could be exploited. The following lemma identifies this hidden matrix, to be exploited in section \ref{sec:Golub-Welsch_method}.

\begin{lemma}\label{lem:orth}
The Hermite transform matrix $T \in \mathbb{R}^{N\times N}$, with the nodes $\mathbf{x}$ taken to be the Gauss--Hermite quadrature nodes, satisfies
\begin{equation}
    T = DQ^T, \qquad T^{-1} = Q D^{-1},
\end{equation}
where $Q \in \mathbb{R}^{N\times N}$ is orthogonal and $D \in \mathbb{R}^{N\times N}$ is diagonal, with
\begin{align*}
    {Q}_{kj} = \frac{\psi_k(x_j)}{\sqrt{N} \left|\psi_{N{-}1}(x_j)\right|}, \qquad {D}_{jj} = \sqrt{N}\left|\psi_{N{-}1}(x_j)\right|.
\end{align*}
\end{lemma}

\begin{proof}
Clearly $T = DQ^T$ by the definition of $D$ and $Q$. What remains to show is that $Q$ is orthogonal. Consider equation \eqref{eqn:Tinv}. It says that ${T}^T {W} {T} = {I}$, where $W$ is a diagonal matrix with elements $W_{kk} = N^{-1}\psi_{N{-}1}(x_k)^{-2}$ by equation \eqref{eqn:ghweights}. This implies that $W^{1/2}T = Q$ is orthogonal.
\end{proof}

\begin{figure}[h!]
\begin{subfigure}{0.495\textwidth}
		\centering
		\includegraphics[width=0.8\linewidth]{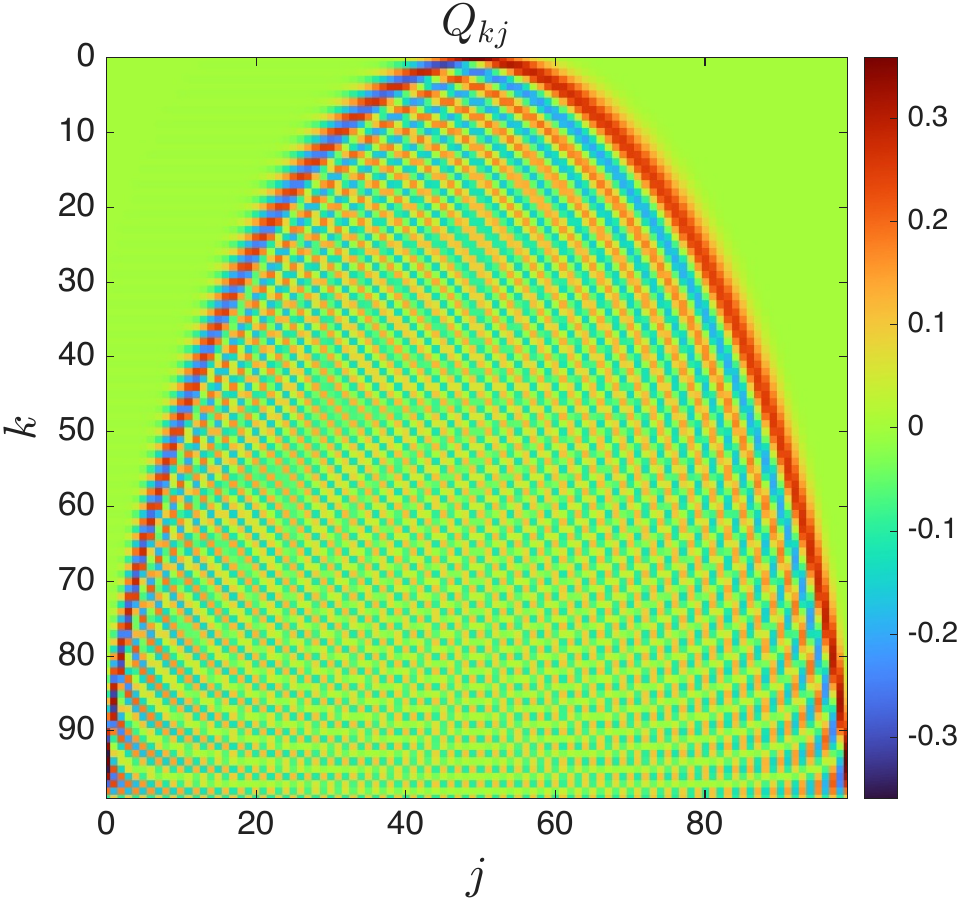}
        \end{subfigure}
        \begin{subfigure}{0.495\textwidth}
		\centering
        \includegraphics[width=0.7\linewidth]{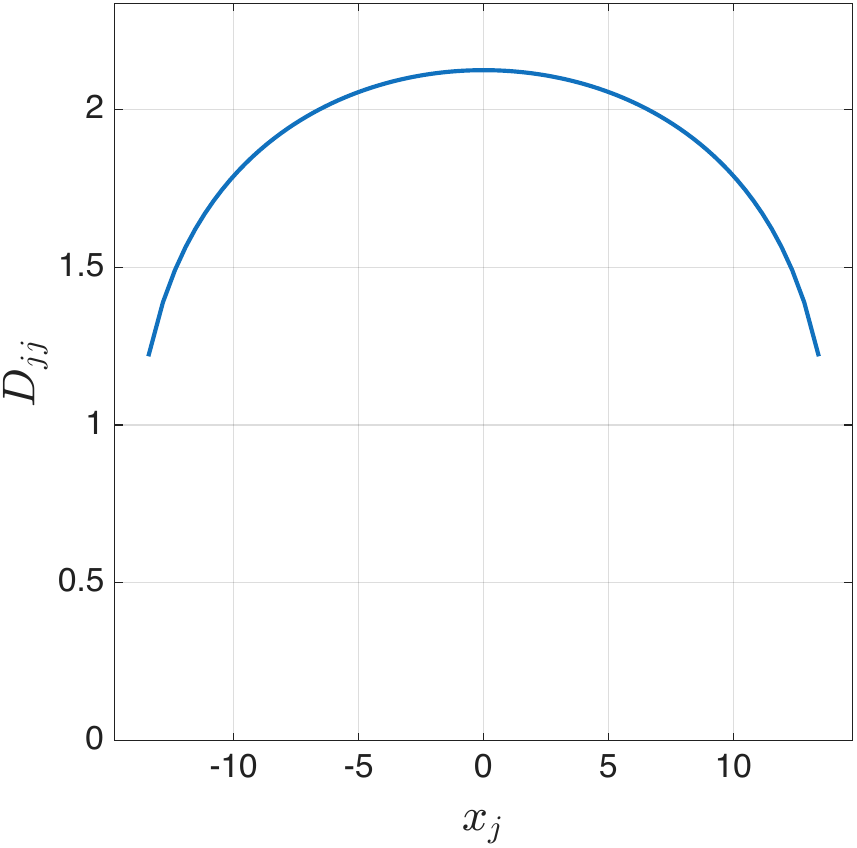}
        \end{subfigure}
    \caption{Depiction of the matrix $Q$ and the diagonal of $D$ for $N = 100$.}
    \label{fig:depictQandd}
\end{figure}

\subsection{New approach via the Golub--Welsch algorithm}\label{sec:Golub-Welsch_method}

The Golub--Welsch algorithm computes Gaussian quadrature nodes and weights via the eigendecomposition of the Jacobi matrix associated with the three term recurrence \cite{golub1969calculation}. In this section, we begin by describing the mathematical foundation of this approach and before providing a summary of the steps taken in Algorithm~\ref{alg:GolubWelschHermiteTransform}. The starting point of the algorithm is the Jacobi matrix $J\in\mathbb{R}^{N\times N}$ associated with the Hermite polynomials
\begin{align}
    J_{i,j} =
\begin{cases}
\sqrt{\dfrac{i}{2}}, & j = i+1, \\
\sqrt{\dfrac{j}{2}}, & i = j+1, \\
0, & \text{otherwise},
\end{cases}
\quad \text{for } i,j = 0,1,\dots,N{-}1. \label{eqn:ghJacobimatrix}
\end{align}

\begin{lemma}\label{lem:jacobi}
    The Jacobi matrix $J\in\mathbb{R}^{N\times N}$ in equation \eqref{eqn:ghJacobimatrix} has eigendecomposition
    $$
    J = Q\Lambda Q^T,
    $$
    where $Q$ is as in Lemma \ref{lem:orth} and $\Lambda = \mathrm{diag}(\mathbf{x})$ where $\mathbf{x}$ is the vector of Gauss--Hermite quadrature nodes.
    \begin{proof}
The $k$th column of the relationship $JQ = Q\Lambda$ follows from equation \eqref{eqn:hermiterec} evaluated at $x = x_k$ and divided by $\sqrt{N}|\psi_{N{-}1}(x_k)|$. The orthogonality of $Q$ proved in Lemma \ref{lem:orth} yields $J = JQQ^T = Q\Lambda Q^T$.
    \end{proof}
\end{lemma}

Lemma \ref{lem:jacobi} shows us that the matrix $Q$ in the orthogonal factorisations of $T$ and $T^{-1}$ can be computed in $\mathcal{O}(N^2)$ operations by calling a tridiagonal eigensolver. To complete the computation, we must compute the diagonal elements of the matrix $D$, which requires computation of $\psi_{N{-}1}(\mathbf{x})$.

While the standard Gauss--Hermite weights $w_k$ can be extracted directly from the first row of the matrix $Q$ \cite{golub1969calculation}, we believe that the weights $D_{jj}$ cannot be extracted directly from the $Q$ or $\Lambda$ --- but we would be glad to be proved wrong. So, we propose to calculate $\psi_{N-1}(\mathbf{x})$ by Clenshaw's algorithm \cite{clenshaw1955note}, \cite[§2.7]{olver2020fast} applied to the expansion \eqref{eqn:finite_expansion} with $\mathbf{c} = (0,0,\ldots,1)^T$ for small $N$ (because this algorithm experiences underflow or overflow for the same values of $N$ as the direct method) and use an asymptotic expansion for the Hermite functions for large $N$.

As discussed in \cite{townsend2016fast}, following \cite[12.7]{NIST:DLMF},
\begin{equation}
    \psi_n(x) = (n!\sqrt{\pi})^{-\tfrac12} U\left( -\frac12\mu^2, \mu t \sqrt{2} \right),
\end{equation}
where $U$ is the parabolic cylinder function, $\mu = \sqrt{2n+1}$ and $t = x/\mu$. For $-1 < t \leq 1$ and $\mu \to +\infty$, we have the asymptotic expansion \cite[12.10.35]{NIST:DLMF},
$$
U\!\left(-\tfrac{1}{2}\mu^2,\ \mu t\sqrt{2}\right)
\sim 2\pi^{1/2}\mu^{1/3} g(\mu)\phi(\zeta)
\left(
\operatorname{Ai}(\mu^{4/3}\zeta)\sum_{s=0}^{\infty}\frac{A_s(\zeta)}{\mu^{4s}}
+
\frac{\operatorname{Ai}'(\mu^{4/3}\zeta)}{\mu^{8/3}}
\sum_{s=0}^{\infty}\frac{B_s(\zeta)}{\mu^{4s}}
\right),
$$
where $\zeta$  satisfies
$\frac{2}{3}(-\zeta)^{3/2}=
\frac{1}{2}\cos^{-1} t - \frac{1}{2}t\sqrt{1-t^2}$,
$\phi(\zeta) = \left(\frac{\zeta}{t^2 - 1}\right)^{1/4}$,  the functions $A_0,A_1,\ldots$, $B_0, B_1,\ldots$ are described in \cite[12.10.42]{NIST:DLMF}, and
$$
g(\mu) = h(\mu)\left(1 + \sum_{s=1}^{\infty} \frac{\phi_s}{\big((1/2)\mu^2\big)^s} \right), \qquad h(\mu) = 2^{-(1/4)\mu^2 - 1/4} \, \mathrm{e}^{-(1/4)\mu^2} \, \mu^{(1/2)\mu^2 - 1/2},
$$
where the coefficients $\phi_s$ are defined by
$$
\Gamma\!\left(\frac{1}{2} + z\right)
\sim \sqrt{2\pi}\, \mathrm{e}^{-z} z^z \sum_{s=0}^{\infty} \frac{\phi_s}{z^s}.
$$
We use this expansion to evaluate $D_{kk}$ for $N \geq 200$ and for $\left\lfloor \frac{N}{2}\right\rfloor \leq k \leq N-1$. These values of $k$ correspond to $x_k \geq 0$. Values of $k$ corresponding to negative $x_k$ are computed by symmetry. In Figure \ref{fig:placeholder}, we provide evidence that the use of this asymptotic expansion to compute $T$ is accurate for $N \geq 200$, but that for $N < 200$ we should use Clenshaw's algorithm.

\begin{figure}[h!]
    \centering
    \begin{subfigure}{0.495\textwidth}
		\centering
		\includegraphics[width=0.99\linewidth]{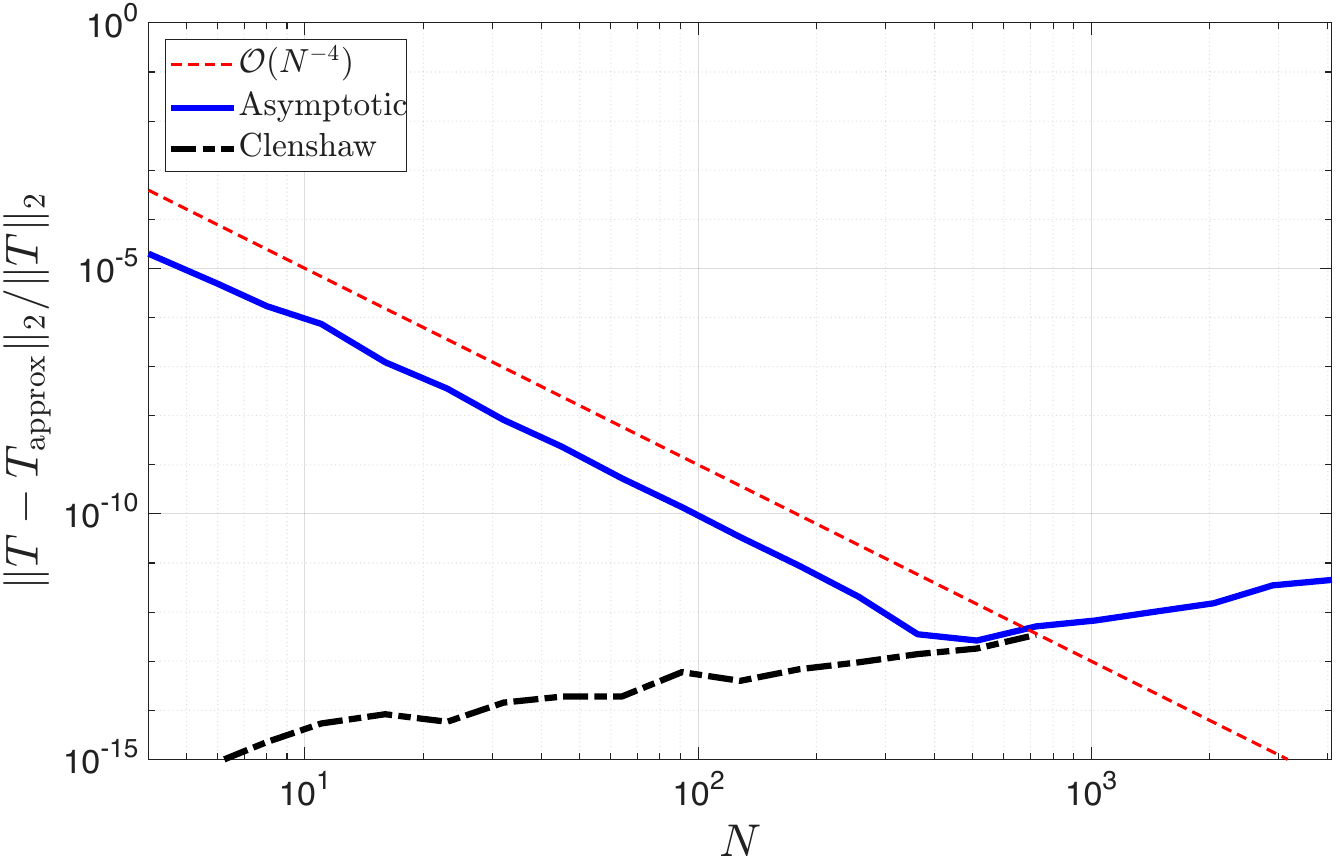}
        \caption{Error in $T$.}
        \end{subfigure}
        \begin{subfigure}{0.495\textwidth}
		\centering
        \includegraphics[width=0.99\linewidth]{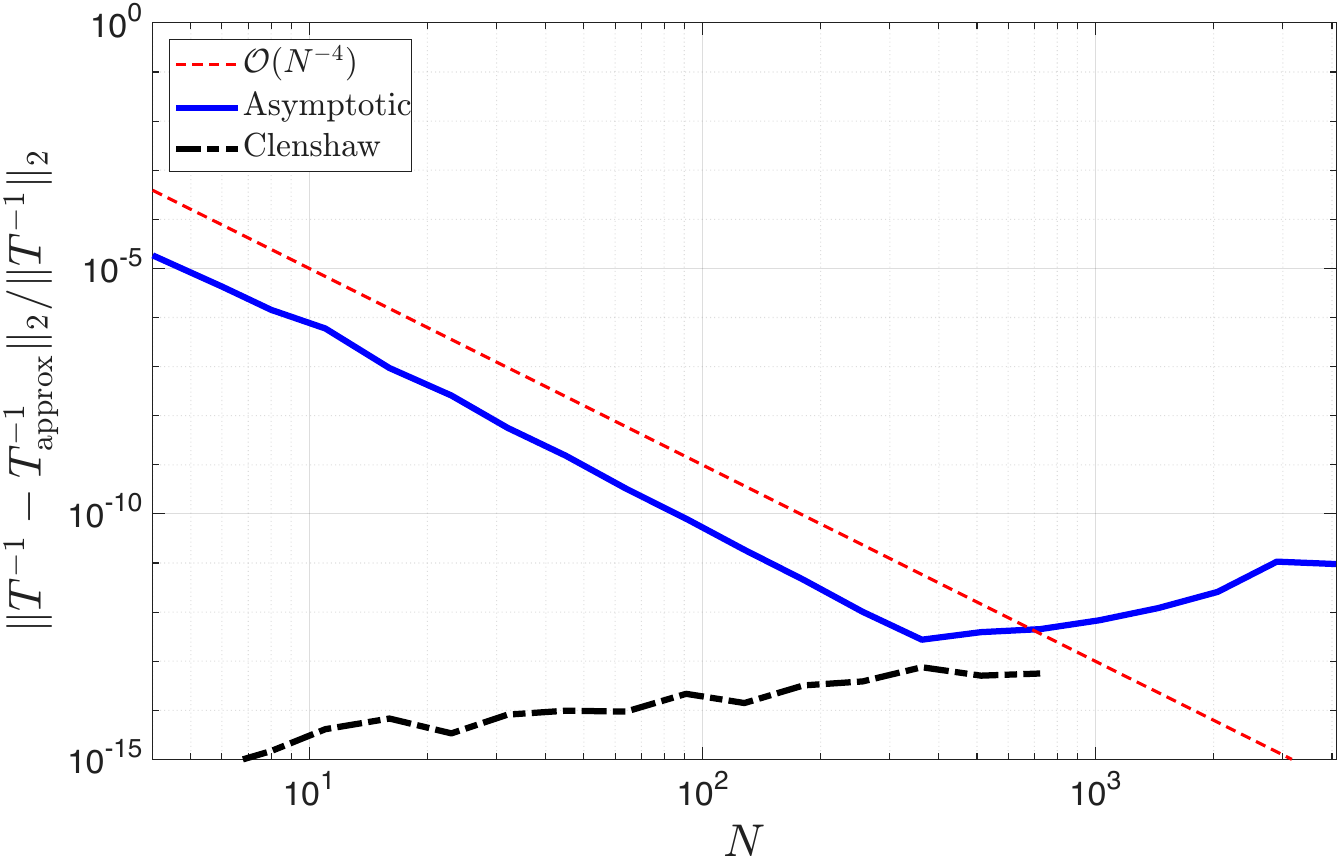}
        \caption{Error in $T^{-1}$.}
        \end{subfigure}
		
    \caption{Error in Algorithm~\ref{alg:GolubWelschHermiteTransform} using Clenshaw's algorithm vs the asymptotic expression described above (cf.~\cite[Fig.~2]{townsend2016fast}). We have excluded errors that are larger than $10^{-1}$. We speculate that the slowly growing error for large $N$ (in the blue lines) is due to the error in the computation of the Airy function and its derivative.}
    \label{fig:placeholder}
\end{figure}
\begin{algorithm}[h!]\caption{Golub--Welsch method for building the Hermite transform factors \(D\) and \(Q\) \\ \textbf{Input:} Integer \(N \ge 1\) \\ \textbf{Output:} \(Q\in\mathbb{R}^{N\times N}\), \(\mathbf{d}\in\mathbb{R}^{N}\) and $\mathbf{x}$ such that \(T = \mathrm{diag}(\mathbf{d}) Q^T\) and $\mathbf{x}$ are the Gauss--Hermite quadrature nodes.}\label{alg:GolubWelschHermiteTransform}
\begin{algorithmic}[1]
\State Assemble the Jacobi matrix \(J\in\mathbb{R}^{N\times N}\) using \eqref{eqn:ghJacobimatrix}
\State Compute the eigendecomposition
$J = Q\, \mathrm{diag}(\mathbf{x})\, Q^T$ (such that $\mathbf{x}$ is sorted)
\For{$j=0,1,2,\ldots,N{-}1$}
\State \(Q_{:,j} = Q_{:,j}\,\mathrm{sign}(Q_{N,j})(-1)^{N+j}\)
\EndFor
\State Compute \(\mathbf{d}\in\mathbb{R}^{N}\) using
\[
(\mathbf{d})_j = \sqrt{N}\,\big|\psi_{N-1}(x_{j})\big|,
\qquad j=0,1,2,\ldots,N{-}1,
\]
where \(\psi_{N-1}\) is evaluated as described in the text above.
\end{algorithmic}
\end{algorithm}

\begin{remark}
    Note that many algorithms for computing the eigendecomposition of a symmetric tridiagonal matrix require $\mathcal{O}(N^3)$ operations in the worst case \cite{demmel2008performance}; however, the Algorithm of Multiple Relatively Robust Representations (MRRR) \cite{dhillon2005glued} in LAPACK (\texttt{STEGR}) requires $\mathcal{O}(N^2)$ operations. This appears to be the algorithm used by MATLAB in our codes.
\end{remark}

\section{Numerical examples}\label{sec:numerical_examples}
In this section we provide a few numerical examples comparing the performance of the aforementioned three algorithms:
\begin{itemize}
    \item ``Direct'': the direct assembly of the Hermite transform matrices using the recurrence \eqref{eqn:hermiterec};
    \item ``Bunck'': the assembly of $\mathbf{d},Q$ using the stable algorithm described in section~\ref{sec:Bunck_method};
    \item ``Golub--Welsch'': our novel method based on the Golub--Welsch algorithm introduced in section~\ref{sec:Golub-Welsch_method}.
\end{itemize}

\subsection{Performance metrics of the stable Hermite transforms}
To begin with we evaluate the assembly time for various sizes of the Hermite transform (see Figure~\ref{fig:assembly_time}). We observe that our novel algorithm achieves significant speed-up over Bunck's method and in particular is comparable in cost than the Direct method.
\begin{figure}[h!]
\centering
		\centering
		\includegraphics[width=0.49\textwidth]{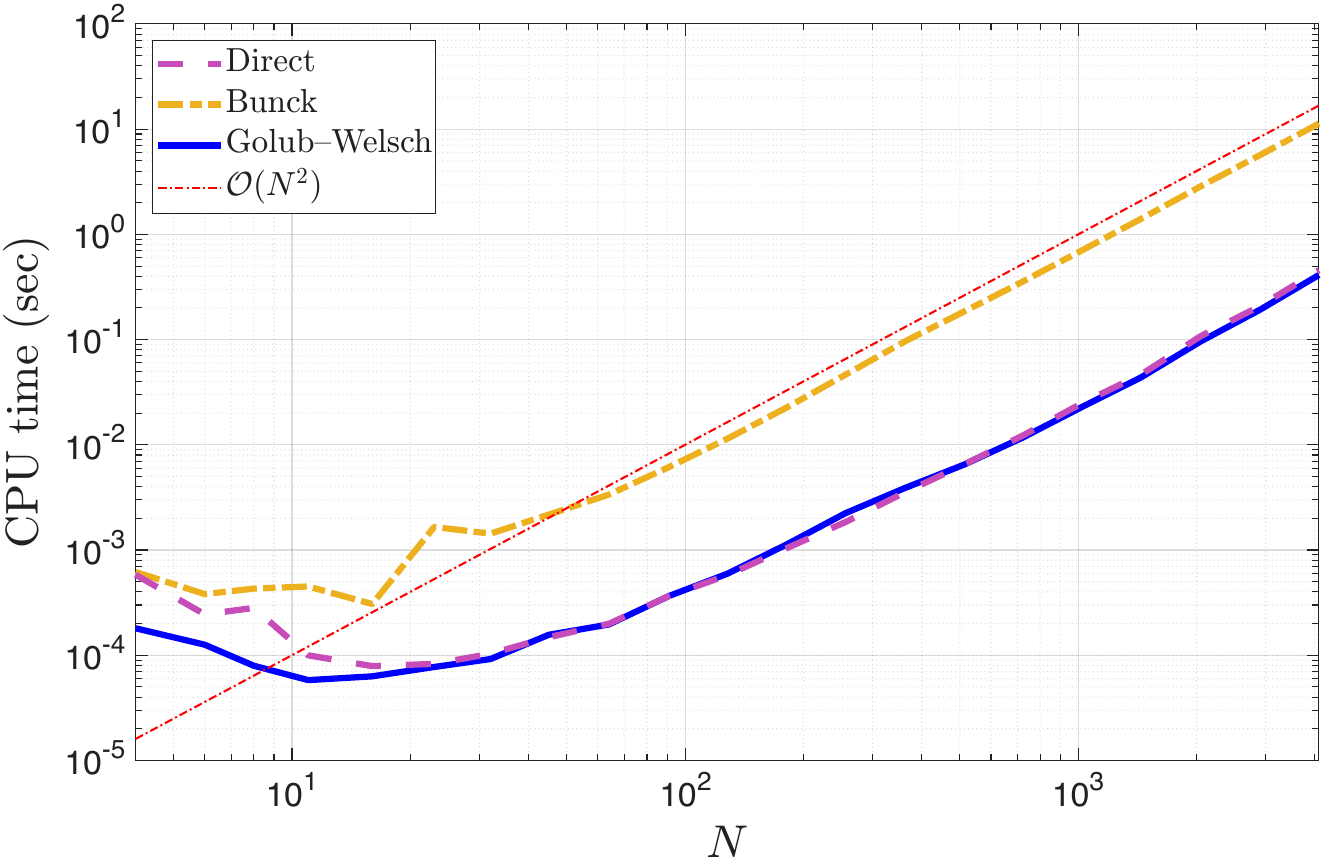}
  \caption{Assembly time for the matrix $T$.}
  \label{fig:assembly_time}
\end{figure}
Next we consider the error in the approximation of the transform matrix, $T$, and its inverse, $T^{-1}$. For this we compute reference values for the corresponding matrices using a high-precision Julia implementation of Bunck's algorithm (accurate to $10^{-34}$). As anticipated we observe the clear instability in the Direct method appearing both in $T$ and $T^{-1}$ {(around $N\approx 766$)} while both Bunck's method and the new Golub--Welsch-based algorithms are accurate and stable for large values of $N$. The important observation in this experiment is that the behaviour of our novel approach (``Golub--Welsch'') does not change as we switch to the asymptotic expression for the $\psi_N$ when $N>200$.

\begin{figure}[h!]
	\begin{subfigure}{0.495\textwidth}
		\centering
		\includegraphics[width=0.92\textwidth]{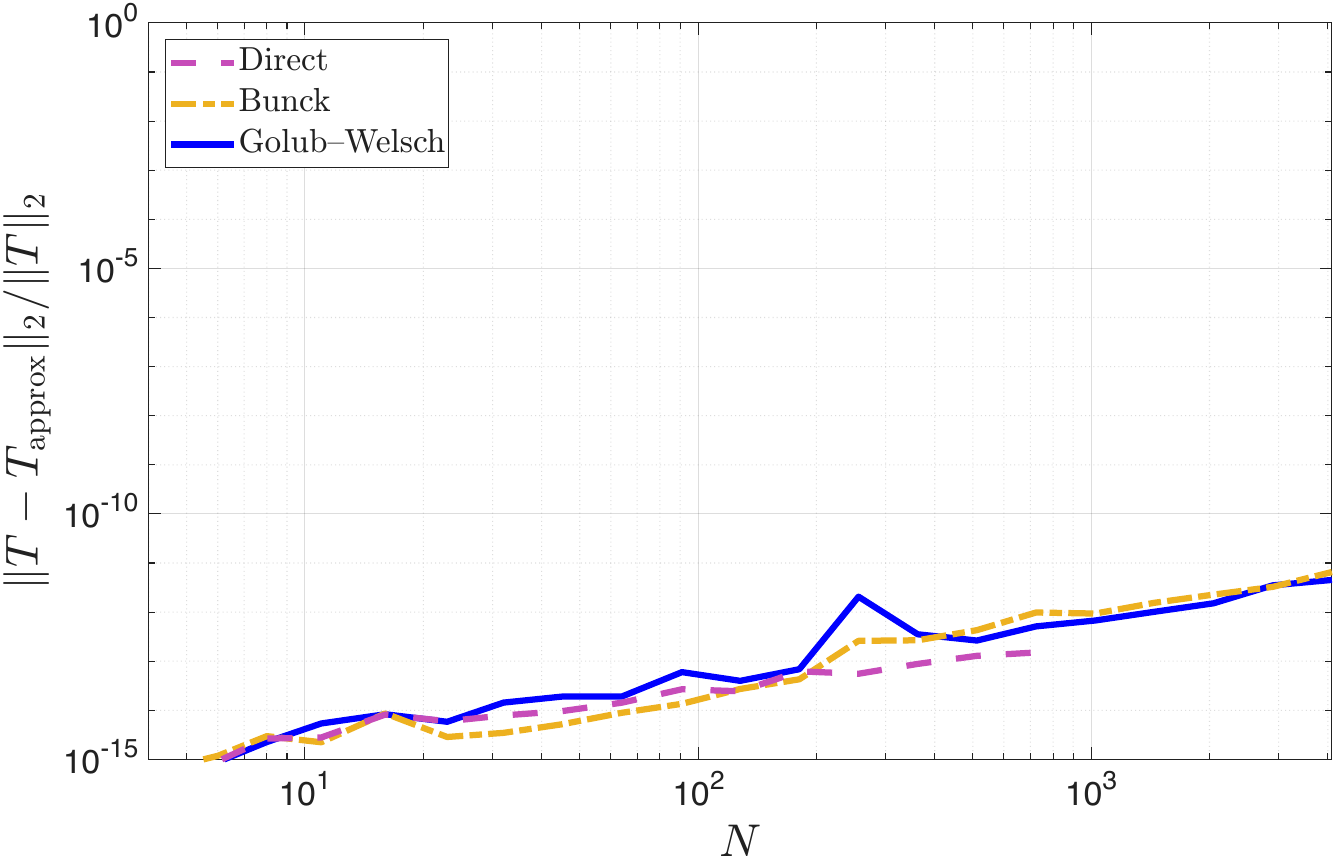}
  \caption{Error in the approximation of $T$.}
	\end{subfigure}
    \begin{subfigure}{0.495\textwidth}
		\centering
		\includegraphics[width=0.92\textwidth]{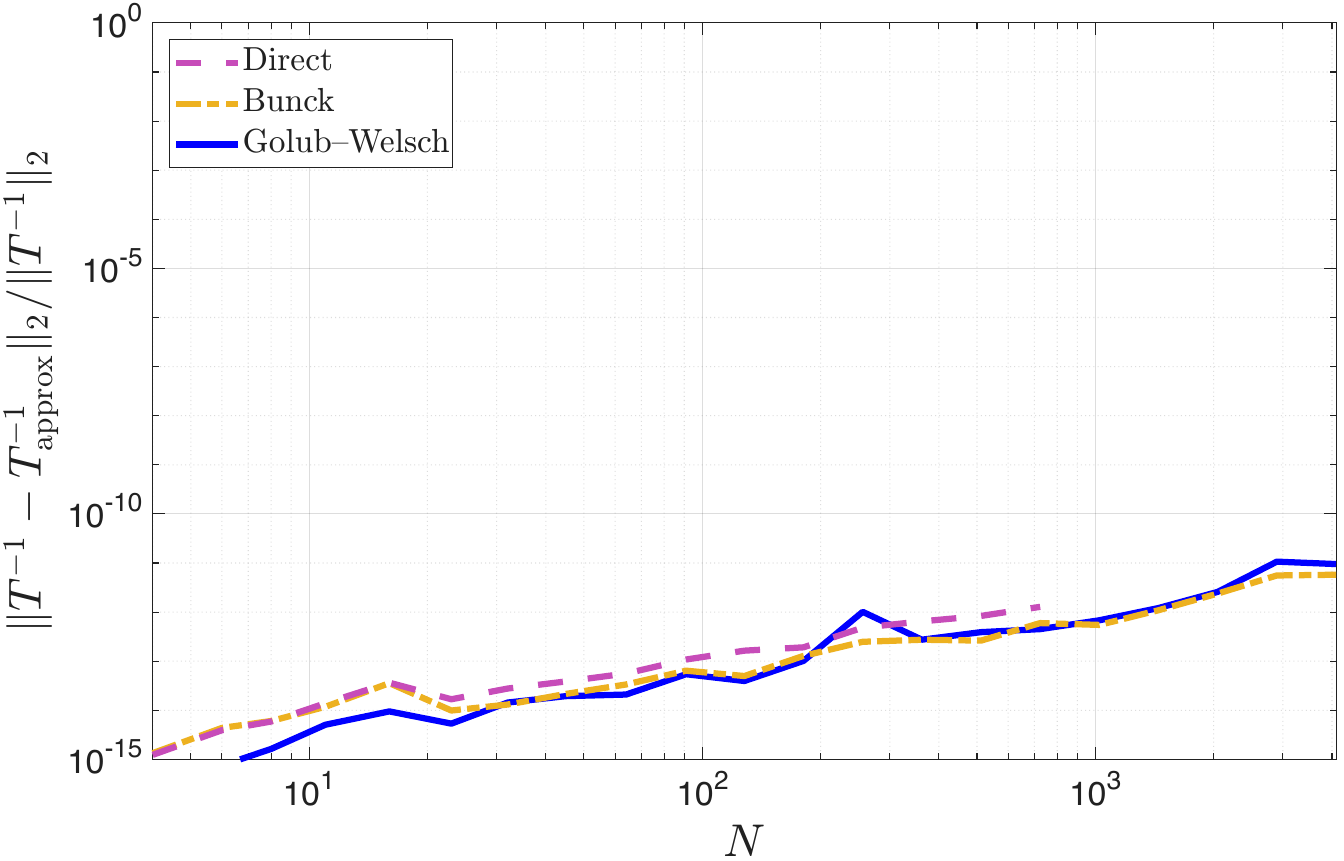}
  \caption{Error in the approximation of $T^{-1}$.}
	\end{subfigure}
	\caption{Error in the transform matrices ($2$-norm). We have excluded errors that are larger than $10^{-1}$.}
	\label{fig:performance_metrics}
\end{figure}


In summary, the performance metrics show that the new approach provides a highly efficient and stable way for assembling the transform matrix, outperforming both Bunck's and the Direct method in these respective metrics.

\subsection{Downstream task: the Gross--Pitaevskii equation}\label{sec:downstream}
To demonstrate the reliable performance of our novel algorithm for computing the Hermite transform we apply it to a downstream task, namely the numerical approximation of the Gross--Pitaevskii equation modelling Bose--Einstein condensation at temperatures below critical condensation \cite{bao2003numerical,gross1961structure,pitaevskii1961vortex} which is given by
\begin{align}\label{eqn:Gross-Pitaevskii_equation}
i \partial_t u(x,t)
=
\left(
-\frac{1}{2}\partial_{xx}
+\frac{1}{2}x^2
+\beta |u(x,t)|^2
\right)u(x,t),
\qquad x\in\mathbb{R}, \; t>0.
\end{align}
In the present example we choose a spatial discretisation based on the Hermite transform, noting that in particular the action of $-\frac{1}{2}\partial_{xx}
+\frac{1}{2}x^2$ on this basis is diagonal:
\begin{align*}
    \left(-\frac{1}{2}\partial_{xx}+\frac{1}{2}x^2\right)\psi_n(x)
    = \left(n+\frac{1}{2}\right)\psi_n(x).
\end{align*}
This makes the Hermite basis a natural starting point for a splitting method based on the two subproblems:
\begin{align*}
    \partial_t u^{(1)} &= -i\left(-\tfrac12\partial_{xx}+\tfrac12 x^2\right)u^{(1)}, \\
    \partial_t u^{(2)} &= -i\beta |u^{(2)}|^2u^{(2)},
\end{align*}
which have exact solutions
\begin{align}\label{eqn:exact_sln1}
    u^{(1)}(t+\tau)
    &= \exp\!\left(-i\tau\left(-\tfrac12\partial_{xx}+\tfrac12 x^2\right)\right)u^{(1)}(t), \\\label{eqn:exact_sln2}
    u^{(2)}(t+\tau)
    &= \exp\!\left(-i\beta\tau |u^{(2)}(t)|^2\right)u^{(2)}(t),
\end{align}
respectively, for a time step $\tau>0$, that can be computed using diagonal operations in Hermite coefficient and physical space. In Figure~\ref{fig:gp_equation} we compare the performance of two approaches for the spatial discretisation (using $N$ Hermite modes), the direct method (``Direct'') and our new Golub--Welsch-based approach (``Golub--Welsch'') from Section~\ref{sec:Golub-Welsch_method}. The initial condition used in the simulation is $u_0(x)=\sqrt{8}\exp(-(x+25)^2/8)\exp(ix/2)$ and our time step is $\tau=10^{-3}$. We observe that in downstream tasks the performance observations from the previous section translate into practice: the direct method leads to instability {around $N\approx 766$ with associated breakdown in the simulation of the Gross--Pitaevskii equation}. In contrast, our novel approach is stable and remains reliable even for large numbers of spatial modes. Note, we observe that even in this simple example a large number of Hermite modes is required for accurate spatial resolution, justifying why using more than $N>766$ modes is reasonable for practical applications.
\begin{figure}[h!]\centering
	\begin{subfigure}{0.5\textwidth}
		\includegraphics[width=0.99\textwidth]{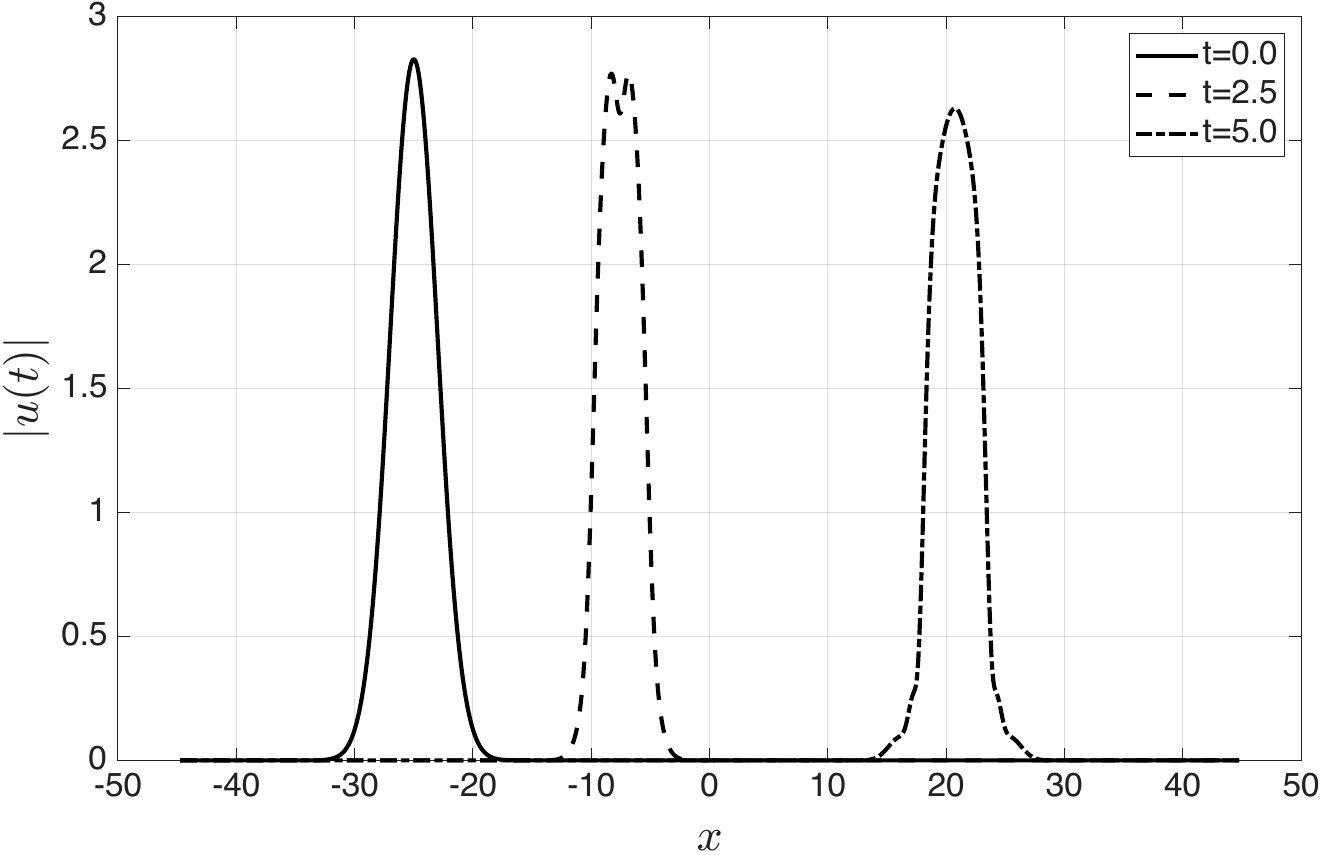}
		\caption{The solution field at $t=0,2.5,5$.}
	\end{subfigure}\begin{subfigure}{0.5\textwidth}
		\includegraphics[width=0.99\textwidth]{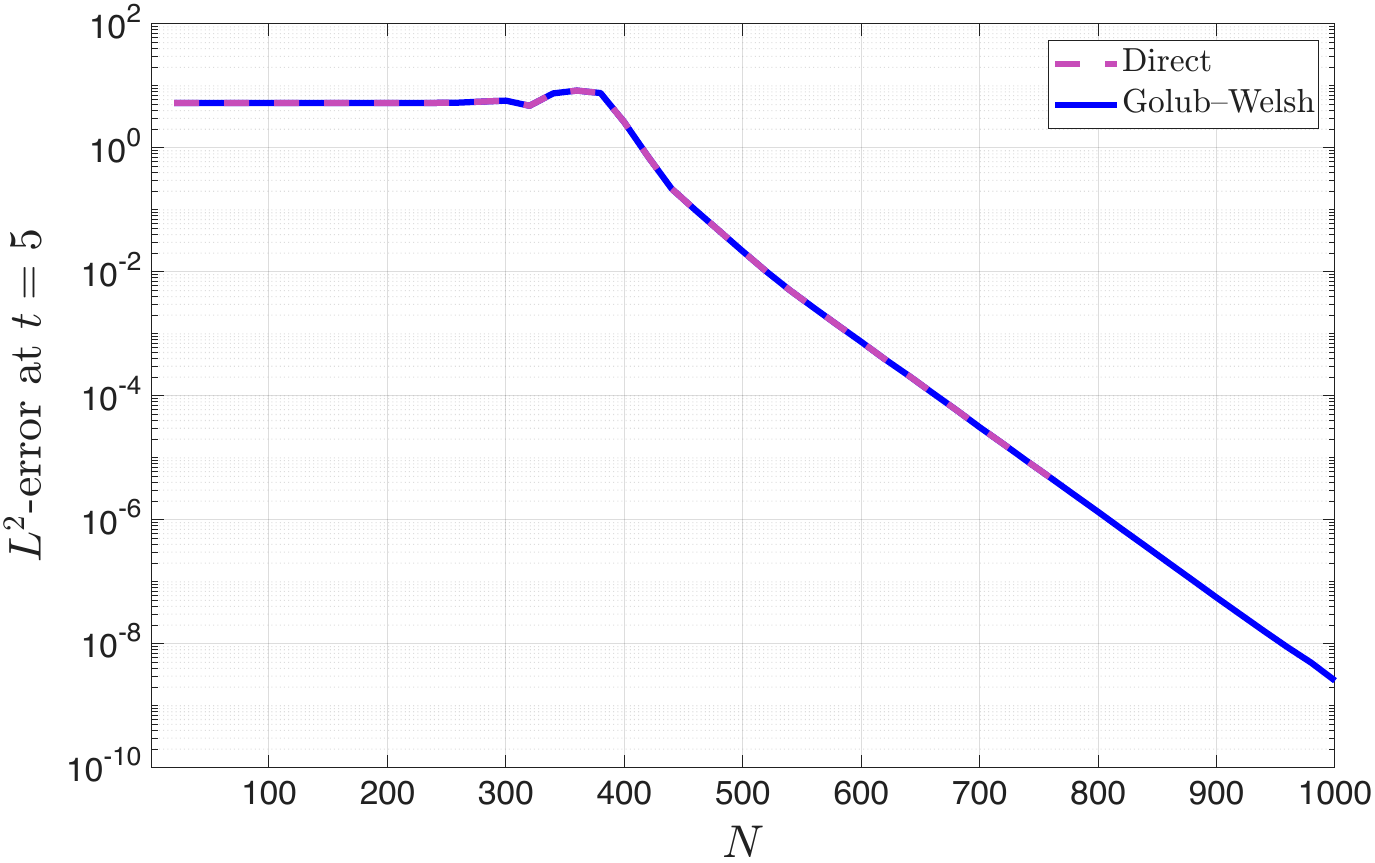}
		\caption{Approximation error at $t=5$.}
        \label{fig:approx_error_gp}
	\end{subfigure}
	\caption{Performance of the Hermite spatial discretisation on the Gross--Pitaevskii equation \eqref{eqn:Gross-Pitaevskii_equation}.}
	\label{fig:gp_equation}
\end{figure}

\section{Conclusions}

In this manuscript we introduced a novel algorithm for the stable computation of Hermite transforms by exploiting the factorisation of the transform matrix in the form $T = DQ^T$, where $Q$ is orthogonal and $D$ is diagonal. We compute these factors via the eigendecomposition of the associated Hermite Jacobi matrix. This perspective connects the computation of Hermite transforms directly to the classical Golub--Welsch framework for Gaussian quadrature and provides a conceptually simple and computationally efficient route to assembling the transform. In addition, we provide a modern presentation of Bunck's algorithm whose work \cite{bunck2009fast} has, until recently, perhaps unjustly received only limited attention in the community.

A certain advantage of the Golub--Welsch approach to computing $T$ and $T^{-1}$ is that we compute the Gauss--Hermite quadrature nodes $\{x_k \}_{k=0}^{N-1}$ for free, whereas the other approaches require these nodes to be precomputed (possibly using the Golub--Welsch algorithm!). We also obtain the Gauss--Hermite quadrature weights (almost) for free, by $w_k = \mathrm{e}^{-x_k^2}/D_{k,k}^2$ (by equation \eqref{eqn:ghweights}) or from $Q$ in the same fashion as the Golub--Welsch algorithm.

Our numerical experiments demonstrate that the proposed approach yields reliable accuracy and improved performance compared with constructing the Hermite transform matrix directly. Moreover, the resulting transforms are sufficiently stable for use in downstream scientific computing tasks, such as the simulation of partial differential equations posed on unbounded domains.

We hope that the community will find the open-source implementations of our algorithm useful in their work (available at \url{https://github.com/marcusdavidwebb/StableHermiteTransforms}).

\section*{Acknowledgements}

The authors would like to thank Sheehan Olver and Mika\"el Slevinsky for interesting discussions related to these algorithms. The authors gratefully acknowledge financial support from the University of Manchester and the University of Cambridge which has facilitated mutual visits that led to this work. MW thanks the Polish National Science Centre (SONATA-BIS-9), project no. 2019/34/E/ST1/00390, for the funding that supported some of this research.

\bibliographystyle{plain}
\bibliography{refs}

@ARTICLE{bunck2009fast,
  title     = "A fast algorithm for evaluation of normalized {H}ermite functions",
  author    = "Bunck, Benjamin F",
  journal   = "BIT Numer. Math.",
  publisher = "Springer Science and Business Media LLC",
  volume    =  49,
  number    =  2,
  pages     = "281--295",
  year      =  2009,
  language  = "en"
}

@article{golub1969calculation,
  title={Calculation of {G}auss quadrature rules},
  author={Golub, Gene H and Welsch, John H},
  journal={Math. Comput.},
  volume={23},
  number={106},
  pages={221--230},
  year={1969}
}

@misc{FastTransforms,
  author       = {R. M. Slevinsky and S. Olver and others},
  title        = {{FastTransforms.jl}},
  howpublished = {\url{https://github.com/JuliaApproximation/FastTransforms.jl}},
  note         = {Version 0.17.0},
  year         = {2025}
}

@misc{FastGaussQuadrature,
  author       = {A. Townsend and S. Olver and others},
  title        = {{FastGaussQuadrature.jl}},
  howpublished = {\url{https://github.com/JuliaApproximation/FastGaussQuadrature.jl}},
  note         = {Version 1.0.2},
  year         = {2024}
}

@article{iserles2023approximation,
  title={Approximation of wave packets on the real line},
  author={Iserles, Arieh and Luong, Karen and Webb, Marcus},
  journal={Constr. Approx.},
  volume={58},
  number={1},
  pages={199--250},
  year={2023},
  publisher={Springer}
}

@article{WANG2019108815,
title = {{Approximation of the {B}oltzmann collision operator based on Hermite spectral method}},
journal = {J. Comput. Phys.},
volume = {397},
pages = {108815},
year = {2019},
author = {Yanli Wang and Zhenning Cai},
}

@article{grohs2017tensor,
  title={Tensor-product discretization for the spatially inhomogeneous and transient {B}oltzmann equation in two dimensions},
  author={Grohs, Philipp and Hiptmair, Ralf and Pintarelli, Simon},
  journal={SMAI J. Comput. Math.},
  volume={3},
  pages={219--248},
  year={2017}
}

@article{townsend2016fast,
  title={{Fast computation of {G}auss quadrature nodes and weights on the whole real line}},
  author={Townsend, Alex and Trogdon, Thomas and Olver, Sheehan},
  journal={IMA J. Numer. Anal.},
  volume={36},
  number={1},
  pages={337--358},
  year={2016},
  publisher={Oxford University Press}
}

@article{olver2020fast,
  title={Fast algorithms using orthogonal polynomials},
  author={Olver, Sheehan and Slevinsky, Richard Mika{\"e}l and Townsend, Alex},
  journal={Acta Numer.},
  volume={29},
  pages={573--699},
  year={2020},
  publisher={Cambridge University Press}
}

@book{amparoseguratemme2007,
author = {Gil, Amparo and Segura, Javier and Temme, Nico M.},
title = {Numerical Methods for Special Functions},
publisher = {Society for Industrial and Applied Mathematics},
year = {2007}
}

@Article{Thalhammer2012,
  author     = {Thalhammer, Mechthild},
  title      = {Convergence analysis of high-order time-splitting pseudospectral methods for nonlinear {S}chr\"{o}dinger equations},
  journal    = {SIAM J. Numer. Anal.},
  year       = {2012},
  volume     = {50},
  number     = {6},
  pages      = {3231--3258},
  issn       = {0036-1429},
  doi        = {10.1137/120866373},
  fjournal   = {SIAM J. Numer. Anal.},
  mrclass    = {65L60 (65L70)},
  mrnumber   = {3022261},
  mrreviewer = {Rossana Vermiglio},
  url        = {https://doi.org/10.1137/120866373},
}

@article{pitaevskii1961vortex,
  title={Vortex lines in an imperfect {B}ose gas},
  author={Pitaevskii, Lev P},
  journal={Sov. Phys. JETP},
  volume={13},
  number={2},
  pages={451--454},
  year={1961}
}

@article{gross1961structure,
  title={Structure of a quantized vortex in boson systems},
  author={Gross, Eugene P},
  journal={Il Nuovo Cimento (1955-1965)},
  volume={20},
  number={3},
  pages={454--477},
  year={1961},
  publisher={Springer}
}

@article{bao2003numerical,
  title={{Numerical solution of the {G}ross--{P}itaevskii equation for Bose--Einstein condensation}},
  author={Bao, Weizhu and Jaksch, Dieter and Markowich, Peter A},
  journal={J. Comput. Phys.},
  volume={187},
  number={1},
  pages={318--342},
  year={2003},
  publisher={Elsevier}
}

@Article{ThCaNe2009,
  author     = {Thalhammer, Mechthild and Caliari, Marco and Neuhauser, Christof},
  title      = {High-order time-splitting {H}ermite and {F}ourier spectral methods},
  journal    = {J. Comput. Phys.},
  year       = {2009},
  volume     = {228},
  number     = {3},
  pages      = {822--832},
  issn       = {0021-9991},
  doi        = {10.1016/j.jcp.2008.10.008},
  fjournal   = {Journal of Computational Physics},
  mrclass    = {65M70},
  mrnumber   = {2477790},
  url        = {https://doi.org/10.1016/j.jcp.2008.10.008},
}

@article{bao2009generalized,
  title={{A generalized-Laguerre--Fourier--Hermite pseudospectral method for computing the dynamics of rotating Bose--Einstein condensates}},
  author={Bao, Weizhu and Li, Hailiang and Shen, Jie},
  journal={SIAM Journal on Scientific Computing},
  volume={31},
  number={5},
  pages={3685--3711},
  year={2009},
  publisher={SIAM}
}

@article{bao2005fourth,
  title={{A fourth-order time-splitting Laguerre--Hermite pseudospectral method for Bose--Einstein condensates}},
  author={Bao, Weizhu and Shen, Jie},
  journal={SIAM Journal on Scientific Computing},
  volume={26},
  number={6},
  pages={2010--2028},
  year={2005},
  publisher={SIAM}
}

@article{gauckler2011convergence,
  title={{Convergence of a split-step Hermite method for the Gross--Pitaevskii equation}},
  author={Gauckler, Ludwig},
  journal={IMA J. Numer. Anal.},
  volume={31},
  number={2},
  pages={396--415},
  year={2011},
  publisher={OUP}
}

@book{griffiths2018introduction,
  title={Introduction to quantum mechanics},
  author={Griffiths, David J and Schroeter, Darrell F},
  year={2018},
  edition={third},
  publisher={Cambridge University Press}
}

@article{walter1977properties,
  title={{Properties of Hermite series estimation of probability density}},
  author={Walter, Gilbert G},
  journal={Ann. Stat.},
  pages={1258--1264},
  year={1977},
  publisher={JSTOR}
}

@article{carlesmaierhofer25,
  title={{On scattering for NLS: rigidity properties and numerical simulations via the lens transform}},
  author={Carles, R{\'e}mi and Maierhofer, Georg},
  journal={Mathematical Models and Methods in Applied Sciences},
  year={2026},
  publisher={World Scientific}
}

@article{banicamaierhoferschratz26,
    title={Computing nonlinear {S}chr\"odinger equations with Hermite functions beyond harmonic traps},
    author={Banica, Valeria and Maierhofer, Georg and Schratz, Katharina},
    journal={arXiv preprint arXiv:2512.20840},
    year={2025}
}

@article{izenman1991review,
  title={Review papers: Recent developments in nonparametric density estimation},
  author={Izenman, Alan Julian},
  journal={J. Am. Stat. Assoc.},
  volume={86},
  number={413},
  pages={205--224},
  year={1991},
  publisher={Taylor \& Francis}
}

@book{szego1939orthogonal,
  title={Orthogonal polynomials},
  author={Szeg\H{o}, Gabor},
  volume={23},
  year={1939},
  publisher={AMS}
}

@misc{NIST:DLMF,
         key = "{\relax DLMF}",
       title = "{\it NIST Digital Library of Mathematical Functions}",
howpublished = "\url{https://dlmf.nist.gov/}, Release 1.2.6 of 2026-03-15",
         url = "https://dlmf.nist.gov/",
        note = "F.~W.~J. Olver, A.~B. {Olde Daalhuis}, D.~W. Lozier, B.~I. Schneider,
                R.~F. Boisvert, C.~W. Clark, B.~R. Miller, B.~V. Saunders,
                H.~S. Cohl, and M.~A. McClain, eds."}

@article{clenshaw1955note,
  title={A note on the summation of {C}hebyshev series},
  author={Clenshaw, Charles W},
  journal={Math. Comput.},
  volume={9},
  number={51},
  pages={118--120},
  year={1955}
}

@article{leibon2008fast,
  title={A fast {H}ermite transform},
  author={Leibon, Gregory and Rockmore, Daniel N and Park, Wooram and Taintor, Robert and Chirikjian, Gregory S},
  journal={Theor. Comput. Sci.},
  volume={409},
  number={2},
  pages={211--228},
  year={2008},
  publisher={Elsevier}
}

@article{townsend2018fast,
  title={Fast polynomial transforms based on {T}oeplitz and {H}ankel matrices},
  author={Townsend, Alex and Webb, Marcus and Olver, Sheehan},
  journal={Math. Comput.},
  volume={87},
  number={312},
  pages={1913--1934},
  year={2018}
}

@article{slevinsky2018use,
  title={On the use of {H}ahn’s asymptotic formula and stabilized recurrence for a fast, simple and stable {C}hebyshev--{J}acobi transform},
  author={Slevinsky, Richard Mika{\"e}l},
  journal={IMA J. Numer. Anal.},
  volume={38},
  number={1},
  pages={102--124},
  year={2018},
  publisher={Oxford University Press}
}

@article{jin2011mathematical,
  title={Mathematical and computational methods for semiclassical {S}chr{\"o}dinger equations},
  author={Jin, Shi and Markowich, Peter and Sparber, Christof},
  journal={Acta Numer.},
  volume={20},
  pages={121--209},
  year={2011},
  publisher={Cambridge University Press}
}

@article{demmel2008performance,
  title={Performance and accuracy of {LAPACK}'s symmetric tridiagonal eigensolvers},
  author={Demmel, James W and Marques, Osni A and Parlett, Beresford N and V{\"o}mel, Christof},
  journal={SIAM J. Sci. Comput.},
  volume={30},
  number={3},
  pages={1508--1526},
  year={2008},
  publisher={SIAM}
}

@article{dhillon2005glued,
  title={Glued matrices and the {MRRR} algorithm},
  author={Dhillon, Inderjit S and Parlett, Beresford N and V{\"o}mel, Christof},
  journal={SIAM J. Sci. Comput.},
  volume={27},
  number={2},
  pages={496--510},
  year={2005},
  publisher={SIAM}
}

\begin{appendix}
\newpage
\section{MATLAB Code}\label{app:matlab_code}\vspace{-0.4cm}
\begin{figure}[!ht]
\begin{lstlisting}[language=matlab, numbers=none]
function [d, Q, x] = initialise_Hermite_transform_Golub_Welsch(N)
%INITIALISE_HERMITE_TRANSFORM_GOLUB_WELSCH
% Builds the orthogonal matrix Q and weight vector d such that
% the coeffs2vals transform is d .* (Q' * cfs) and
% the val2coeffs transform is Q * (vals ./ d).
% See Webb--Maierhofer 2026.

J = diag(sqrt(.5:.5:(N-1)/2), 1) + diag(sqrt(.5:.5:(N-1)/2), -1);
[Q, x] = eig(J, 'vector');                % see Golub-Welsch 1969
Q = Q .* sign(Q(N,:)) .* (-1).^(N+1:2*N); % enforce signs of final row

d = sqrt(N) * abs(herm_func(N{-}1, x(floor(N/2)+1:end)));
d = [d(end:-1:1); d((1+mod(N,2)):end)];
end

function val = herm_func(N, x)
% Evaluates the degree-N Hermite function (for x in [0, sqrt(2N+1)])
if N <= 200 % use Clenshaw's algorithm
    val = exp(-x.^2/2) / pi^(1/4);
    val1 = zeros(size(x));
    for k = N:-1:1
        val2 = val1; val1 = val;
        val = x .* val1 * sqrt(2/k) - val2 / sqrt(1 + 1/k);
    end
else % use Airy asymptotics, DLMF (12.10.35), accurate for N > 200
    mu2 = 2*N+1; t = x/sqrt(mu2);      % 12.10.1
    theta = acos(t); t2 = t.^2;
    eta = (theta - t.*sqrt(1-t2))/2;   % 12.10.23
    zeta = -(3*eta/2).^(2/3);          % 12.10.39
    phi = (zeta./(t2-1)).^(1/4);       % 12.10.40

    % 12.10.43:
    a1 = 15/144; b1 = -7/5*a1;
    a2 = 5*7*9*11/2/144^2; b2 = -13/11*a2;
    a3 = 7*9*11*13*15*17/6/144^3;

    % 12.10.9:
    u1 = (t2-6).*t/24; u2 = ((-9*t2 + 249).*t2 + 145)/1152;
    u3 = ((((-4042*t2+18189).*t2-28287).*t2-151995).*t2-259290).*t/414720;

    % 12.10.42:
    phi6 = phi.^6; A0 = 1; B0 = -(phi6.*u1+a1)./zeta.^2;
    A1 = ((phi6.*u2 + b1*u1).*phi6 + b2)./zeta.^3;
    B1 = -(((phi6.*u3 + a1*u2).*phi6 + a2*u1).*phi6 + a3)./zeta.^5;

    % 12.10.35:
    Airy0 = airy(mu2^(2/3)*zeta);
    Airy1 = airy(1, mu2^(2/3)*zeta);
    val = Airy0.*(A0+A1/mu2^2) + (Airy1/mu2^(4/3)).*(B0+B1/mu2^2);

    % 12.10.14:
    g = (((-4027/4976640/mu2 + 1003/103680)/mu2 + 1/576)/mu2 - 1/24)/mu2 + 1;
    g = g*exp(-gammaln(N+1)/2+(mu2/4-1/12)*log(mu2)-(mu2-3)*log(2)/4-mu2/4);
    val = (pi^(1/4) * g) * phi .* val;
end
end
\end{lstlisting}\vspace{-0.2cm}
\caption{MATLAB code of the novel algorithm described in section~\ref{sec:Golub-Welsch_method}.}
\end{figure}

\end{appendix}
\end{document}